\documentclass[12pt, dvipdfmx]{article}
\usepackage{ascmac}
\usepackage{amsmath,amssymb,nccmath}
\usepackage{mdframed, tikz,tikz-cd,graphicx}
\usepackage[all]{xy}
\usepackage{comment}

\usepackage{theorem}
\theoremstyle{break}
\theorembodyfont{\rmfamily}
\usepackage{latexsym}
\def\oqed{\hfill $\Box$}

\newtheorem{defi}{定義}[section]

\newtheorem{thm}[defi]{Theorem}
\newtheorem{prop}[defi]{Proposition}

\newtheorem{lem}[defi]{Lemma}

\newtheorem{proof}{\bf{Proof.}}

\newcommand{\phiGamma}[1]{
(\varphi ,\Gamma_{#1} )
}
\newcommand{\phiqGamma}[1]{
(\varphi_{q} ,\Gamma_{#1} )
}
\newcommand{\Robba}[1]{
\mathbf{B}^{\dag}_{F, #1}
}
\newcommand{\RobbaB}{
\mathbf{B}^{\dag}_{\widehat{F}_{{\rm unr}}}
}
\newcommand{\catMod}[1]{
{\rm Mod}_{/\mathbf{B}_{F, #1}^{\dag}}^{\varphi_{q}, \Gamma_{#1}, \acute{e}t}
}

\title{Overconvergent Lubin--Tate $(\varphi, \Gamma)$-modules for different uniformizers}
\author{Yuta Saito
\thanks{ysaito@ms.u-tokyo.ac.jp}
}

\begin{document}

\maketitle

\begin{abstract}
\noindent
A\textsc{bstract}. 
Let $F$ be a non-archimedean local field. 
The construction of Lubin--Tate $\phiqGamma{}$-modules attached to $p$-adic representations of $G_F$ depends on the choice of a uniformizer of $F$. 
In this paper, we give a description of a functor which relates categories of overconvergent Lubin--Tate $\phiqGamma{}$-modules for different uniformizers. 
Further, we study this functor more explicitly for $2$-dimensional trianguline representations. 
\end{abstract}

\section{Introduction}

Let $K$ be a finite extension of $\mathbf{Q}_{p}$. Fontaine introduced the theory of cyclotomic $\phiGamma{}$-modules in \cite{Fo} to classify $p$-adic representation of $G_{K}$ where $\Gamma = {\rm Gal}(K_{\infty}/K)$ which $K_{\infty}$ is the field generated by $p$-power roots of $1$ over $K$ called the cyclotomic extension. 
He uses a big ring called the ring of periods to attach  $\phiGamma{}$-modules to $p$-adic representations of $G_{K}$. 
There is an important fact that these cyclotomic $\phiGamma{}$-modules are always overconvergent which is proved by Cherbonnier and Colmez \cite{CC}; this allows us to relate Fontaine's $\phiGamma{}$-modules and $p$-adic Hodge theory.

Cyclotomic $\phiGamma{}$-modules play an essential role in the construction of the $p$-adic local Langlands correspondence for ${\rm GL}_{2}(\mathbf{Q}_{p})$ (see \cite{Breuil}). 
In order to generalize this correspondence to ${\rm GL}_{2}(F)$, where $F$ is a finite extension of $\mathbf{Q}_{p}$, it seems necessary to extend the theory of cyclotomic $\phiGamma{}$-module in some manner. 
One of the expected way is considering Lubin--Tate $\phiGamma{}$-modules for which $\Gamma = {\rm Gal}(K_{\infty, \pi}/K)$, where $F \subset K$ and $K_{\infty, \pi}$ is generated by the torsion points of a Lubin--Tate group attached to a uniformizer $\pi$ of $F$. 
For sake of simplicity, we assume $K=F$ in this paper. 
We can attach to each $F$-linear representatioin of $G_{F}$ a Lubin--Tate $(\varphi, \Gamma)$-module over a certain field $\mathbf{B}_{K, \pi}$ (see \cite{KR}). 
However, these $\phiGamma{}$-modules are usually not overconvergent contrary to the cyclotomic case (see \cite{FX}). 
Relate to this issue, Berger developed an approach by considering $\phiGamma{}$-modules with coefficients in rings of pro-analytic vectors and  proved that $F$-analytic representations are overconvergent (see \cite{Ber}).

We have another problem of the theory of Lubin--Tate $\phiGamma{}$-modules that a category of Lubin--Tate $\phiGamma{}$-modules depends on the choice of a uniformizer of $F$. 
Let $V$ be an $F$-linear representation of $G_{F}$ where $F$ is a finite extentsion of $\mathbf{Q}_{p}$. 
The construction of Lubin--Tate $\phiGamma{}$-modules, especially the construction of a Lubin--Tate extension, depends on the choice of a uniformizer $\pi$ of $F$ so that the overconvergency of $V$ may depend on the choice of $\pi$. 
We say that $V$ is $\pi$-overconvergent if $V$ is overconvergent for Lubin--Tate $\phiGamma{}$-modules for the uniformizer $\pi$. 
If $V$ is $\pi$-overconvergent, then we can attach to $V$ a Lubin--Tate $\phiGamma{}$-module $D_{\pi}^{\dag}$ over a subfield $\mathbf{B}_{F, \pi}^{\dag}$ of $\mathbf{B}_{F, \pi}$ consisting of overconvergent elements. 
Thus if we want to compare $\pi$-overconvergence and $\pi'$-overconvergence for different uniformizers $\pi$ and $\pi'$ of $F$, we should compare Lubin--Tate $\phiGamma{}$-modules over $\mathbf{B}_{F, \pi}^{\dag}$ and Lubin--Tate $\phiGamma{}$-modules over $\mathbf{B}_{F, \pi'}^{\dag}$. Let $H_{\pi} = {\rm Gal}(\overline{F} / F_{\infty, \pi})$ where $\overline{F}$ is an algebraic closure of $F$ and $F_{\infty, \pi}$ is a Lubin--Tate extension of $F$ for a uniformizer $\pi$. 
To compare Lubin--Tate $\phiGamma{}$-modules for different uniformizers, we consider the functor
\[
\mathbf{M}_{\pi, \pi'} : D_{\pi}^{\dag} \mapsto (\mathbf{B}^{\dag} \otimes_{\mathbf{B}_{F, \pi}^{\dag}} D_{\pi}^{\dag})^{H_{\pi'}},
\]
where $\mathbf{B}^{\dag}$ be one of the big fields of $p$-adic periods such that $(\mathbf{B}^{\dag})^{H_{\pi}} = \mathbf{B}_{F, \pi}^{\dag}$.
This functor has the following compatibility which is proved by easy calculation from the result of \cite{FX} (Proposition \ref{compatibilityofM}). 

\begin{prop}
If $V$ is $\pi$-overconvergent and $\pi'$-overconvergent, then corresponding Lubin--Tate $\phiGamma{}$-modules $D_{\pi}^{\dag}$, $D_{\pi'}^{\dag}$ satisfy $D_{\pi}^{\dag} = \mathbf{M}_{\pi, \pi'}(D_{\pi}^{\dag})$. 
\end{prop}

If we want to calculate $\mathbf{M}_{\pi, \pi'}$, it is hard to do with present formulation of $\mathbf{M}_{\pi, \pi'}$. 
In order to carry out this calculation easier, we introduce a new ring of periods $\mathbf{B}_{\widehat{F}_{{\rm unr}}}^{\dag}$ which is smaller than $\mathbf{B}^{\dag}$ and give simpler description of $\mathbf{M}_{\pi, \pi'}$. 
The ring $\mathbf{B}_{\widehat{F}_{\rm unr}}^{\dag}$ is the field of overconvergent power serieses with coefficient in $\widehat{F}_{{\rm unr}}$, where $\widehat{F}_{{\rm unr}}$ is the $\pi$-adic completion of the maximal unramified extension of $F$. 
We take a topological generator $\sigma_{\pi'}$ of ${\rm Gal}(F_{{\rm unr}, \infty}/F_{\infty, \pi'})$ such that $\sigma_{\pi'} |_{F_{{\rm unr}}} = \varphi_{q}$, where $F_{{\rm unr}, \infty} = F_{{\rm unr}} \cdot F_{\infty, \pi'}$. 
Let $\chi_{\pi}$ be a Lubin--Tate character for $\pi$ and let $u \in \mathcal{O}_{F}^{\times}$ be the unit of $\mathcal{O}_{F}$ such that $\pi' = u\pi$. 
Our result is the following (see Theorem \ref{Maintheorem}).

\begin{thm} \label{introMainthm}
We have $\mathbf{M}_{\pi, \pi'}(D_{\pi}^{\dag}) = (\mathbf{B}_{\widehat{F}_{{\rm unr}}}^{\dag} \otimes_{\mathbf{B}_{F, \pi}^{\dag}} D_{\pi}^{\dag})^{\sigma_{\pi'} \otimes \chi_{\pi}^{-1}(u) = 1}$.
\end{thm} 

If $V$ is $\pi$-overconvergent and $D_{\pi}^{\dag}$ denotes the corresponding overconvergent Lubin--Tate $\phiGamma{}$-module, we can check the $\pi'$-overconvergency of $V$ by calculating $\mathbf{M}_{\pi, \pi'}(D_{\pi}^{\dag})$. 
We focus on carrying out the calculation of $\mathbf{M}_{\pi, \pi'}$ for $2$-dimensional triangulin representations and we conclude that overconvergency is independet of the choice of uniformizers in this case.

\vspace{12pt}
\noindent
A\textsc{cknowledgments}: 
I would like to thank my academic supervisor, Associate Professor Naoki Imai, for many useful discussions and crucial suggestions, and for spending time reviewing early drafts of this paper.

\section{Lubin--Tate Extensions}

Throughout this article, $F$ is a finite extension of $\mathbf{Q}_{p}$ with ring of integers $\mathcal{O}_{F}$ and residue field $k_{F}$. 
Let $q = p^{h}$ be the cardinality of $k_{F}$ and let $e$ be the ramification index of $F$ over $\mathbf{Q}_{p}$. 
We identify $k_{F}$ with $\mathbf{F}_{q}$. 

Let $\pi$ be a uniformizer of $F$. 
Let $\mathcal{G}_{\pi}$ be a Lubin--Tate formal group attached to $[\pi]_{\pi}(T) = \pi T + T^{q}$. For $a \in \mathcal{O}_{F}$, let $[a]_{\pi}(T)$ denote the power series that gives the multiplicaion-by-$a$ map on $\mathcal{G}_{\pi}$. Let $F_{n, \pi} = F(\mathcal{G}_{\pi}[\pi^{n}])$, and let $F_{\infty , \pi} = \bigcup_{n\geq 1}F_{n, \pi}$. 
Let $H_{\pi} = {\rm Gal}(\overline{F} / F_{\infty , \pi})$ and $\Gamma_{\pi} = {\rm Gal}(F_{\infty , \pi} / F)$. 
By Lubin--Tate theory (see \cite[Chapter VI, section 3.4]{ANT}), $\Gamma_{\pi}$ is isomorphic to $\mathcal{O}_{F}^{\times}$ via the Lubin--Tate character $\chi_{\pi} : \Gamma_{\pi} \to \mathcal{O}_{F}^{\times}$. 

Let $u_{\pi, 0} = 0$ and for each $n\geq 1$, let $u_{\pi, n} \in \overline{\mathbf{Q}}_p$ be such that $[\pi]_{\pi}(u_{\pi, n}) = u_{\pi, n-1}$ with $u_{\pi, 1} \neq 0$.

These Lubin--Tate extensions and its Galois groups are dependent on the choice of the uniformizer $\pi$. 
Let $\pi$, $\pi' $ be two uniformizers of $F$ and let $u \in \mathcal{O}_{F}^{\times}$ be the unit of $\mathcal{O}_{F}$ such that $\pi' = u\pi$. 
There is a relationship between two Lubin--Tate extensions of $F$ as following (see \cite[Chapter VI, Section 3.7]{ANT}). 

\begin{lem} \label{lem-formalgroup}
Let $F_{{\rm unr}}$ be the maximal unramified extension of $F$ and $\widehat{F}_{{\rm unr}}$ be the completion of $F_{{\rm unr}}$. 
Let $\sigma_{q} \in {\rm Gal}(F_{{\rm unr}} / F)$ be the arithmetic Frobenius automorphisms and extend it to $\widehat{F}_{{\rm unr}}$ by continuity. 
Then there exists a power series $\phi \in \mathcal{O}_{\widehat{F}_{{\rm unr}}}[[X]]$ with $\phi(X) \equiv \varepsilon X \mod (X^{2})$ where $\varepsilon$ is a unit of $\mathcal{O}_{\widehat{F}_{{\rm unr}}}$, such that
\begin{enumerate}
\item $\phi$ is a morphism from $\mathcal{G}_{\pi}$ to $\mathcal{G}_{\pi'}$ as formal $\mathcal{O}_F$-modules;
\item $^{\sigma_{q}}\phi = \phi \circ [u]_{\pi}$. 
\end{enumerate}
\end{lem}

By local class field theory, we have the next proposition (see \cite[Chapter VI]{ANT}). 

\begin{prop} \label{reciprocitymap}
Let $F_{{\rm unr}, \infty} = F_{{\rm unr}} \cdot F_{\infty, \pi}$. 
Then this is independet of the choice of a uniformizer $\pi$ and there uniquely exists a canonical homomorqhism
\[
r : F^{\times} \to {\rm Gal}(F_{{\rm unr}, \infty}/F)
\]
which satisfies
\begin{enumerate}
\item for any uniformizer $\pi \in F$, $r(\pi)$ is the identity on the $F_{\infty, \pi}$ and is equal to $\sigma_{q}$ on $F_{{\rm unr}}$; 
\item for any $u \in \mathcal{O}_{F}^{\times}$, $r(u)$ is equal to $\chi_{\pi}^{-1}(u^{-1})$ on $F_{\infty, \pi}$ for any uniformizer $\pi \in F$ and is the identity on $F_{{\rm unr}}$. 
\end{enumerate}
The homomorphism $r$ is called \textit{the reciprocity map}. 
\end{prop}

\section{Rings of $p$-adic Periods}

We recall various rings of $p$-adic periods from \S 3 of \cite{Ber}. 
In this section, we fix a uniformizer $\pi$ of $F$. 
Let $F_{\infty, \pi}$ be the Lubin--Tate extension for $\pi$ of $F$ constructed in \S 1. 
Let 
\[
\widetilde{\mathbf{E}}^{+} = \{(x_0 , x_1 , ...), \text{with } x_n \in \mathcal{O}_{\mathbf{C}_p }/\pi \text{ and } x_{n+1}^q = x_n \text{ for all } n\geq 0 \}.
\] 
This ring is endowed with the valuation ${\rm val}_{\mathbf{E}}(\cdot)$ defined by ${\rm val}_{\mathbf{E}}(x) = \lim_{n \to +\infty}$ $q^{n}{\rm val}_{p}(\hat{x}_n)$, where $\hat{x}_{n} \in \mathcal{O}_{\mathbf{C}_{p}}$ lifts $x_{n}$. 
The ring $\widetilde{E}^{+}$ is complete for ${\rm val}_{\mathbf{E}}(\cdot)$.  
Let $\{u_{\pi, n}\}_{n\geq 0}$ be as in \S 1. 
Then $\overline{u}_{\pi}=(\overline{u}_{\pi,0}, \overline{u}_{\pi. 1}, ...) \in \widetilde{\mathbf{E}}^{+}$ and ${\rm val}_{\mathbf{E}}(\overline{u}_{\pi}) = q/(q-1)e$. 
Let $\widetilde{\mathbf{E}}$ be the fraction field of $\widetilde{\mathbf{E}}^{+}$. 
Note that there is a canonical inclusion  $\overline{\mathbf{F}}_{q} \hookrightarrow \widetilde{\mathbf{E}}^{+}$ such that $a \mapsto (a_0 , a_1 , ...)$ which $a_n = [a^{q^{-n}}]$ mod $\pi \in W_{F}(\overline{\mathbf{F}}_{q})/\pi \simeq \mathcal{O}_{\widehat{F}_{{\rm unr}}}/\pi$ for all $n \geq 0$ and for any $a\in \overline{\mathbf{F}}_{q}$, where $W_{F}(\cdot)$ denotes the functor $\mathcal{O}_{F} \otimes_{\mathcal{O}_{F_{0}}}  W(\cdot)$ of $F$-Witt vectors. 

Let $\widetilde{\mathbf{A}}^{+} = W_{F}(\widetilde{\mathbf{E}}^{+})$, and let $\widetilde{\mathbf{B}}^{+} = \widetilde{\mathbf{A}}^{+}[1/\pi]$. These rings are preserved by the Frobenius map $\varphi_q = {\rm Id} \otimes \varphi^{h}$, where $\varphi$ is the canonical Frobenius map of Witt ring $W(\widetilde{\mathbf{E}}^{+})$. 
Every element of $\widetilde{\mathbf{B}}^{+}[1/[\overline{u}_{\pi}]]$ can be written as $\sum_{k \gg -\infty}\pi^{k}[x_{k}]$, where $\{ x_{k} \}_{k \in \mathbf{Z}}$ is a bounded sequence of $\widetilde{\mathbf{E}}$. 
For $r \geq 0$, define a valuation $V(\cdot, r)$ on $\widetilde{\mathbf{B}}^{+}[1/[\overline{u}_{\pi}]]$ by 
\[
V(x, r) = \inf_{k\in\mathbf{Z}}\left( \frac{k}{e} + \frac{p-1}{pr}{\rm val}_{\mathbf{E}}(x_{k})\right) \quad {\rm if} \  x = \sum_{k \gg \infty} \pi^{k}[x_{k}].
\]
If $I$ is a closed interval of $[0,+\infty)$, then let $V(x,I) = \inf_{r \in I} V(x,r)$. 
The ring $\widetilde{\mathbf{B}}^{I}$ is defined as the completion of $\widetilde{\mathbf{B}}^{+}[1/[\overline{u}_{\pi} ]]$ for the valuation $V(\cdot, I)$ if $0 \notin I$. 
If $I = [0, r]$, then $\widetilde{\mathbf{B}}^{I}$ is the completion of $\widetilde{\mathbf{B}}^{+}$ for $V(\cdot, I)$. Let $\widetilde{\mathbf{B}}^{\dag} = \cup_{r\gg 0} \widetilde{\mathbf{B}}^{[r,+\infty)}$. 

By \cite[Lemma 1.2]{KR}, there uniquely exists $u_{\pi} \in \widetilde{\mathbf{A}}^{+}$ whose image in $\widetilde{\mathbf{E}}^{+}$ is $\overline{u}_{\pi}$ and such that $\varphi_{q}(u_{\pi}) = [\pi]_{\pi}(u_{\pi})$ and $\gamma(u_{\pi}) = [\chi_{\pi}(\gamma)]_{\pi}(u_{\pi})$ if $\gamma \in \Gamma_{\pi}$.

\begin{prop} \label{prop-coordinatetransformation}
Let $\pi$ and $\pi'$ be uniformizers of F and let $\phi$ be a power series defined as Lemma \ref{lem-formalgroup}. 
Let $\{ u_{\pi', n} \}_{n \geq 0}$ and $u_{\pi'}$ be objects constructed above using $\pi'$ instead of $\pi$. 
Then $\{ u_{\pi', n} \}_{n \geq 0}$ coincides with $\{ ^{\sigma_{q}^{-n}}\phi(u_{\pi, n}) \}_{n \geq 0}$ and $\phi(u_{\pi}) = u_{\pi'}$. 
\end{prop}

\begin{proof}
Remember the definition of the inclusion $\overline{\mathbf{F}}_{q} \hookrightarrow \widetilde{\mathbf{E}}^{+}$. For $a \in \mathcal{O}_{\widehat{F}_{{\rm unr}}}$, 
\[
\sigma_{q}^{-n}(a) \mod \pi = a^{q^{-n}} \mod \pi
\]
in $\mathcal{O}_{\widehat{F}_{{\rm unr}}}/\pi \simeq \overline{\mathbf{F}}_{q}$.
Hence, the image of $\phi(u_{\pi})$ in $\widetilde{\mathbf{E}}^{+}$ is 
\[
( \phi({u}_{\pi, 0}) \text{ mod } \pi, {}^{\sigma_{q}^{-1}}\phi({u}_{\pi, 1}) \text{ mod } \pi , ...).
\]
By Lemma \ref{lem-formalgroup}, $^{\sigma_{q}^{-n}}\phi(u_{\pi, n}) = \phi\circ [u]_{\pi}^{-n}(u_{\pi, n})$. Hence, 
\begin{align}
[\pi']_{\pi'}( &^{\sigma_{q}^{-n}}\phi(u_{\pi, n})) = \phi([\pi']_{\pi}\circ[u]_{\pi}^{-n}u_{\pi, n}) \notag \\
  &= \phi([u]^{-(n-1)}_{\pi}u_{\pi, n-1}) =  {}^{\sigma_{q}^{-(n-1)}}\phi(u_{\pi, n-1}). \notag
\end{align} 
In addition, we can get $\phi(u_{\pi, 0}) = 0$ and $^{\sigma_{q}^{-1}}\phi(u_{\pi, 1}) \neq 0$ by the definition of $\phi$ and invertibility of $\phi$.
Hence, $\{ u_{\pi', n} \}_{n \geq 0}$ coincides with $\{ ^{\sigma_{q}^{-n}}\phi(u_{\pi, n}) \}_{n \geq 0}$. 
By \cite[Lemma 1.2]{KR}, the characterization of $u_{\pi}$ is that it is the lift of $\{ \overline{u}_{\pi, n}\}_{n \geq 0}$ in $\widetilde{A}^{+}$ and satisfies $\varphi_{q}(u_{\pi}) = [\pi]_{\pi}(u_{\pi})$, so we get 
\[
\varphi_{q}(\phi(u_{\pi})) =\  ^{\sigma_{q}}\phi(\varphi_{q}(u_{\pi})) = \phi([u]_{\pi}\circ[\pi]_{\pi}(u_{\pi})) = \phi([\pi']_{\pi}(u_{\pi})) = [\pi']_{\pi'}(\phi(u_{\pi})).
\] 
Accordingly, $\phi(u_{\pi}) = u_{\pi'}$ for $\{ u_{\pi', n} \}_{n \geq 0} = \{ ^{\sigma_{q}^{-n}}\phi(u_{\pi, n}) \}_{n \geq 0}$. 
\oqed \end{proof}

Let $\mathbf{A}_{F, \pi}$ be the $\pi$-adic completion of $\mathcal{O}_{F}[\![ u_{\pi} ]\!][1/u_{\pi}]$ and $\mathbf{B}_{F, \pi}$ be the fraction field of $\mathbf{A}_{F, \pi}$. 
Then $\mathbf{A}_{F, \pi}$ is a complete discrete valuation ring with uniformizer $\pi$ and residue field $\mathbf{F}_{q}((\overline{u}_{\pi})) \subset \widetilde{E}^{+}$. 
Let $\mathbf{B}^{\prime}_{\pi}$ be the maximal unramified extention of $\mathbf{B}_{F, \pi}$ inside $\widetilde{\mathbf{A}}^{+}$ and let $\mathbf{B}$ be the its $\pi$-adic completion. 
Note that $\widehat{F}_{{\rm unr}} \subset \mathbf{B}$; therefore, $\mathbf{B}$ does not depend on the choice of the uniformizer $\pi$ by Proposition \ref{prop-coordinatetransformation}. 

By \cite[Lemma 1.4]{KR} and its proof, we have the next lemma. 

\begin{lem} \label{lem-FoN}
The residue field of $\mathcal{O}_{\mathbf{B}}$ is a separable closure of $\mathbf{F}_{q}((\overline{u}_{\pi}))$. There is a natural isomorphism
\[
{\rm Gal}(\mathbf{B}^{\prime}_{\pi} / \mathbf{B}_{F, \pi}) \simeq {\rm Gal}(\overline{F} / F_{\infty, \pi}) = H_{\pi}
\]
and this isomorphism satisfies that ${\rm Gal}(\overline{F} / F_{\infty, \pi})$ acts on the residue field of $\mathbf{B}^{\prime}_{\pi}$ as the subset of $\widetilde{\mathbf{E}}^{+}$. 
\end{lem}

From this lemma, we have $\mathbf{B}^{H_{\pi}} = \mathbf{B}_{F, \pi}$.

For $\rho > 0$, let 
\[
\rho' = \rho e \cdot \frac{p}{p-1} \cdot \frac{q-1}{q}. 
\]
Let $I$ be a subinterval of $(1 , +\infty )$, and let $f(Y) = \sum_{k \in \mathbf{Z}}a_{k}Y^{k}$ be a power series with $a_{k} \in F$ and such that ${\rm val}_{p}(a_{k}) + k/\rho' \rightarrow +\infty$ when $|k| \rightarrow +\infty$ for all $\rho \in I$.

The  series $f(u_{\pi})$ is converges in $\widetilde{\mathbf{B}}^{I}$ and we let $\mathbf{B}^{I}_{F, \pi}$ denote the set of $f(u_{\pi})$ where $f(Y)$ is as above. It is a subring of $\widetilde{\mathbf{B}}^{I}_{F, \pi} = (\widetilde{\mathbf{B}}^{I})^{H_{\pi}}$, which is stable under the action of $\Gamma_{\pi}$. The Frobenius map gives rise to a map $\varphi_{q} : \mathbf{B}^{I}_{F, \pi} \to \mathbf{B}^{qI}_{F, \pi}$. Let $\mathbf{B}^{\dag , r}_{{\rm rig},  F, \pi}$ denote the ring $\mathbf{B}^{[r , +\infty )}_{F, \pi}$. 
Let $\mathbf{B}^{\dag , r}_{F, \pi}$ denote the set of $f(u_{\pi}) \in \mathbf{B}^{\dag , r}_{{\rm rig}, F, \pi}$ such that in addition $\{a_{k}\}_{k \in \mathbf{Z}}$ is a bounded sequence.  
Let $\mathbf{B}^{\dag}_{F, \pi} = \bigcup_{r \gg 0} \mathbf{B}^{\dag , r}_{F, \pi}$. This is stable under the action of the Frobenius map $\varphi_{q}$.

Let $\RobbaB$ be the ring defined as above replacing $F$ by $\widehat{F}_{{\rm unr}}$. By Proposition \ref{prop-coordinatetransformation}, this does not depend on the choice of the uniformizer $\pi$ and we get the canonical inclusion $\Robba{\pi} \hookrightarrow \RobbaB$ for any uniformizer $\pi$ of $F$.

\section{$(\varphi, \Gamma)$-modules}

For $R = \mathbf{B}_{F, \pi}, \mathbf{B}_{F, \pi}^{\dag}$, a $(\varphi_{q}, \Gamma_{\pi})$-module over $R$ is a free $R$-module of finite rank $D$ with continuous semilinear actions of $\varphi_{q}$ and $\Gamma_{\pi}$ commuting with each other such that $\varphi_{q}$ sends a basis of $D$ to a basis of $D$. 
When $R = \mathbf{B}_{F, \pi}$, we say that $D$ is \textit{\'etale} if $D$ has a $\varphi_{q}$-stable $\mathbf{A}_{F, \pi}$-lattice $M$ such that the linear map $\varphi_{q}^{*}M \to M$ is an isomorphism. 
When $R = \mathbf{B}_{F, \pi}^{\dag}$, we say that $D$ is \textit{\'etale} if $\mathbf{B}_{F, \pi} \otimes_{\mathbf{B}_{F. \pi}^{\dag}} D$ is \'etale. Let ${\rm Mod}_{/R}^{\varphi_{q}, \Gamma_{\pi}, \acute{e}t}$ be the category of \'etale $(\varphi_{q}, \Gamma_{\pi})$-modules over $R$. 

Let ${\rm Rep}_{F} G_{F}$ be the category of finite-dimentional representations of $G_{F}$ over $F$. For any $V \in {\rm Rep}_{F}G_{F}$, put ${\rm D}_{\mathbf{B}_{F, \pi}}(V) = (\mathbf{B} \otimes_{F} V)^{H_{\pi}}$. For any $D_{\pi} \in {\rm Mod}_{/\mathbf{B}_{F, \pi}}^{\varphi_{q}, \Gamma_{\pi}, \acute{e}t}$, put ${\rm V}(D_{\pi}) = (\mathbf{B} \otimes_{\mathbf{B}_{F, \pi}} D)^{\varphi_{q} = 1}$. We have the following theorem (see \cite[Theorem 1.6]{KR}). 

\begin{thm} \label{thm-KR}
The functor $V \mapsto {\rm D}_{\mathbf{B}_{F, \pi}}(V)$ and $D_{\pi} \mapsto {\rm V}(D_{\pi})$ give rise to mutually inverse equivalence of categories between ${\rm Rep}_{F} G_{F}$ and ${\rm Mod}_{/\mathbf{B}_{F, \pi}}^{\varphi_{q}, \Gamma_{\pi}, \acute{e}t}$. 
\end{thm}

We say that $D_{\pi} \in {\rm Mod}_{\mathbf{B}_{F, \pi}}^{\varphi_{q}, \Gamma_{\pi}, \acute{e}t}$ is $\pi$-overconvergent if there exists a basis of $D_{\pi}$ in which the matrices of $\varphi_{q}$ and of all $\gamma \in \Gamma_{\pi}$ have entries in $\mathbf{B}_{F, \pi}^{\dag}$. This basis generates an object $D_{\pi}^{\dag}$ of ${\rm Mod}_{/\mathbf{B}_{F, \pi}^{\dag}}^{\varphi_{q}, \Gamma_{\pi}, \acute{e}t}$. 
We have next proposition from \cite[Proposition 1.5]{FX}.

\begin{prop} \label{prop-overconvergentcatequivalence}
\begin{enumerate}
\item If $D_{\pi}^{\dag}$ is an \'etale $(\varphi_{q}, \Gamma_{\pi})$-module over $\mathbf{B}_{F, \pi}^{\dag}$, then
\[
{\rm V}(\mathbf{B}_{F, \pi} \otimes_{\mathbf{B}_{F, \pi}^{\dag}} D_{\pi}^{\dag}) = (\mathbf{B}^{\dag} \otimes_{\mathbf{B}_{F, \pi}^{\dag}} D_{F, \pi}^{\dag})^{\varphi_{q} = 1}.
\]
\item The functor $V \mapsto (\mathbf{B}^{\dag} \otimes_{F} V)^{H_{\pi}}$ and $D_{\pi}^{\dag} \mapsto {\rm V}(\mathbf{B}_{F, \pi} \otimes_{\mathbf{B}_{F, \pi}^{\dag}} D_{\pi}^{\dag})$ are mutually quasi-inverse equivalences of the category of overconvergent $F$-representation of $G_{F}$ and ${\rm Mod}_{/\mathbf{B}_{F, \pi}^{\dag}}^{\varphi_{q}, \Gamma_{\pi}, \acute{e}t}$. 
\end{enumerate}
\end{prop}

\begin{lem}
Let $\mathbf{B}^{\dag} = \widetilde{\mathbf{B}}^{\dag} \cap \mathbf{B}$. If $D_{\pi}$ is $\pi$-overconvergent, then $D_{\pi}^{\dag} = (\mathbf{B}^{\dag} \otimes_{F} {\rm V}(D_{\pi}))^{H_{\pi}}$ and $D_{\pi}^{\dag}$ is uniquely determined. 
\end{lem}

\begin{proof}
By the $\pi$-overconvergency and Proposition \ref{prop-overconvergentcatequivalence}, ${\rm V}(D_{\pi}) = (\mathbf{B}^{\dag} \otimes_{\mathbf{B}_{F, \pi}^{\dag}} D_{\pi}^{\dag})^{\varphi_{q} = 1}$. By the $\pi$-overconvergency and Theorem \ref{thm-KR}, we have ${\rm dim}_{F} {\rm V}(D_{\pi}) = {\rm dim}_{\mathbf{B}_{F, \pi}^{\dag}} D_{\pi}^{\dag}$. Hence, 
\[
\mathbf{B}^{\dag} \otimes_{F} {\rm V}(D_{\pi}) = \mathbf{B}^{\dag} \otimes_{F} (\mathbf{B}^{\dag} \otimes_{\mathbf{B}_{F, \pi}^{\dag}} D_{\pi}^{\dag})^{\varphi_{q} = 1} = \mathbf{B}^{\dag} \otimes_{\mathbf{B}_{F, \pi}^{\dag}} D_{\pi}^{\dag}.
\]
We have $(\mathbf{B}^{\dag})^{H_{\pi}} = \mathbf{B}_{F, \pi}^{\dag}$ and $H_{\pi}$ acts trivially on $D_{\pi}^{\dag}$.  Consequently, 
\[
(\mathbf{B}^{\dag} \otimes_{F} {\rm V}(D_{\pi}))^{H_{\pi}} = (\mathbf{B}^{\dag} \otimes_{\mathbf{B}_{F, \pi}^{\dag}} D_{\pi}^{\dag})^{H_{\pi}} = D_{\pi}^{\dag}. 
\]
\oqed\end{proof}

We say that $V \in {\rm Rep}_{F}G_{F}$ is $\pi$-overconvergent if ${\rm D}_{\mathbf{B}_{F, \pi}}(V) \in \catMod{\pi}$ is $\pi$-overconvergent. For different uniformizers $\pi$ and $\pi'$ of $F$, we want to compare the $\pi$-overconvergence and $\pi'$-overconvergence of $V$. Let $\mathbf{M}_{\pi, \pi'}$ be the functor from $\catMod{\pi}$ to $\catMod{\pi'}$ such that $\mathbf{M}_{\pi, \pi'}(D_{\pi}^{\dag}) = (\mathbf{B}^{\dag} \otimes_{\mathbf{B}_{F, \pi}^{\dag}} D_{\pi}^{\dag})^{H_{\pi'}}$. First, we have the following proposition.

\begin{prop} \label{compatibilityofM}
We assume that $V$ is $\pi$-overconvergent and $\pi'$-overconvergent. Let $D_{\pi}^{\dag}$ (\textit{resp.} $D_{\pi'}^{\dag}$) be the corresponding object of ${\rm Mod}_{/\mathbf{B}_{F, \pi}^{\dag}}^{\varphi_{q}, \Gamma_{\pi}, \acute{e}t}$ (\textit{resp.} $\catMod{\pi'}$). Then $\mathbf{M}_{\pi, \pi'}(D_{\pi}^{\dag}) = D_{\pi'}^{\dag}$.
\end{prop} 

\begin{proof}
$V = (\mathbf{B}^{\dag} \otimes_{\mathbf{B}_{F, \pi}^{\dag}} D_{\pi}^{\dag})^{\varphi_{q} = 1}$ by Proposition \ref{prop-overconvergentcatequivalence}, Hence, 
\begin{align}
D_{\pi'}^{\dag} &= (\mathbf{B}^{\dag} \otimes_{F} V)^{H_{\pi'}} = \left( \mathbf{B}^{\dag} \otimes_{F} (\mathbf{B}^{\dag} \otimes_{\mathbf{B}_{F, \pi}^{\dag}} D_{\pi}^{\dag})^{\varphi_{q} = 1} \right)^{H_{\pi'}} \notag \\
&= (\mathbf{B}^{\dag} \otimes_{\mathbf{B}_{F, \pi}^{\dag}} D_{\pi}^{\dag})^{H_{\pi'}} =  \mathbf{M}_{\pi, \pi'}(D_{\pi}^{\dag}). \notag
\end{align}
\oqed \end{proof}

If we can show that $\mathbf{M}_{\pi, \pi'}$ and $\mathbf{M}_{\pi', \pi}$ are mutually quasi-inverse equivalences of $\catMod{\pi}$ and $\catMod{\pi'}$, we have the equivalence of $\pi$-overconvergency and $\pi'$-overconvergency of $V \in {\rm Rep}_{F}G_{F}$. 
The next aim is to describe $\mathbf{M}_{\pi, \pi'}(D_{\pi}^{\dag})$ using $\mathbf{B}_{\widehat{F}_{{\rm unr}}}^{\dag}$. 
Let $F_{{\rm unr}, \infty} = F_{{\rm unr}} \cdot F_{\infty, \pi}$, which does not depend on the choice of $\pi$ because it is the maximal abelian extension of F from local class field theory. 
Note that $H_{F_{{\rm unr}}} = {\rm Gal}(\overline{F}/F_{{\rm unr}, \infty})$ acts trivially on $D_{\pi}^{\dag}$. 
Combining Lemma \ref{galoistheoryofringofperiods} we have 
\begin{align}
\mathbf{M}_{\pi, \pi'}(D_{\pi}^{\dag}) &= \left( (\mathbf{B}^{\dag} \otimes_{\mathbf{B}_{F, \pi}^{\dag}} D_{\pi}^{\dag})^{H_{F_{{\rm unr}}}} \right) ^{{\rm Gal}(F_{{\rm unr}, \infty} / F_{\infty, \pi'})} \notag \\
&= (\mathbf{B}_{\widehat{F}_{{\rm unr}}}^{\dag} \otimes_{\mathbf{B}_{F, \pi}^{\dag}} D_{\pi}^{\dag})^{{\rm Gal}(F_{{\rm unr}, \infty} / F_{\infty, \pi'})}. \notag
\end{align}

\begin{lem} \label{galoistheoryofringofperiods}
We have $(\mathbf{B}^{\dag})^{H_{F_{{\rm unr}}}} = \mathbf{B}_{\widehat{F}_{{\rm unr}}}^{\dag}$.
\end{lem}

\begin{proof}
We follow the argument in the proof of \cite[Proposition 3.8]{FO}. 

For an unramified extention $K/F$, let $\mathbf{B}_{K, \pi}$ be the fraction field of the $\pi$-adic completion of $\mathcal{O}_{K}[\![u_{\pi}]\!][1/u_{\pi}]$. 
By the definition of $\mathbf{B}^{\dag}$, it suffices to show $\mathbf{B}^{H_{F_{{\rm unr}}}} = \mathbf{B}_{\widehat{F}_{{\rm unr}}}$.

First, we show $(\mathbf{B}^{\prime}_{\pi})^{H_{F_{{\rm unr}}}} = \mathbf{B}_{F_{{\rm unr}}, \pi}$. 
Let $F_{n}$ be the $n$-dimentional unramified extension of $F$ for $n \geq 1$. Let $F_{n, \infty, \pi} = F_{n} \cdot F_{\infty, \pi}$ and $H_{n, \pi} = {\rm Gal}(\overline{F} / F_{n, \infty, \pi})$. 
By the definition, $\mathbf{B}_{F_{n}, \pi}$ is an unramified extension of $\mathbf{B}_{F, \pi}$. 
The residue field of $\mathbf{B}_{F_{n}, \pi}$ is $\mathbf{F}_{q^{n}}((\overline{u}_{\pi}))$ on which $H_{n, \pi}$ acts trivially. 
Hence, $H_{n, \pi}$ acts trivially on $\mathbf{B}_{F_{n}, \pi}$ by Lemma \ref{lem-FoN} and then we have $[\mathbf{B}_{F_{n}, \pi} : \mathbf{B}_{F, \pi}] = n$. 
Because $[(\mathbf{B}^{\prime}_{\pi})^{H_{n, \pi}}: \mathbf{B}_{F, \pi}] = [F_{n, \infty, \pi}: F_{\infty, \pi}] = n$, we have $(\mathbf{B}^{\prime}_{\pi})^{H_{n, \pi}} = \mathbf{B}_{F_{n}, \pi}$. 
As a result, $(\mathbf{B}^{\prime}_{\pi})^{H_{F_{{\rm unr}}}} = \mathbf{B}_{F_{{\rm unr}}, \pi}$. 

For any $\alpha \in \mathbf{B}^{H_{F_{{\rm unr}}}}$, we choose a sequence of $\alpha_{n} \in \mathbf{B}_{\pi}^{\prime}$ such that $v(\alpha - \alpha_{n}) \geq n$, where $v$ is a non-archemedian valuation on $\mathbf{B}_{\pi}^{\prime}$. 
Then it follows that $v(g(\alpha_{n})-\alpha_{n}) \geq \min \{ v( g(\alpha_{n} - \alpha)), v(\alpha_{n} - \alpha) \} \geq n$ for any $g \in H_{F_{{\rm unr}}}$. 
Applying \cite[Proposition 1]{Ax}, we can find $a_{n} \in \mathbf{B}_{F_{{\rm unr}, \pi}}$ such that $v(\alpha_{n} - a_{n}) \geq n - \frac{p}{(p-1)^{2}}v(p)$. 
Thus we have $\alpha = \lim_{n \to \infty} a_{n} \in \mathbf{B}_{\widehat{F}_{{\rm unr}}}$. 
Hence, $\mathbf{B}^{H_{F_{{\rm unr}}}} = \mathbf{B}_{\widehat{F}_{{\rm unr}}}$. 

Consequently, $(\mathbf{B}^{\dag})^{H_{F_{{\rm unr}}}} = \mathbf{B}_{\widehat{F}_{{\rm unr}}}^{\dag}$. 
\oqed\end{proof}

Recall the reciprocity map $r$ in Lemma \ref{reciprocitymap}. 
Because the restriction map ${\rm Gal}(F_{{\rm unr}, \infty} / F_{\infty, \pi}) \to {\rm Gal}(F_{{\rm unr}}/F)$ induces isomorphism, $r(\pi)$ is a topological generator of ${\rm Gal}(F_{{\rm unr}, \infty} / F_{\infty, \pi})$. 
Thus if we want to calculate $\mathbf{M}_{\pi, \pi'}(D_{\pi}^{\dag})$, find the elements in $\RobbaB \otimes_{\Robba{\pi}} D_{\pi}^{\dag}$ fixed by $r(\pi')$. 
Let $u \in \mathcal{O}_{F}^{\times}$ be the element which satisfies $\pi' = u\pi$. 
Then we have $r(\pi') = r(u)r(\pi)$, thus $r(\pi') |_{F_{\infty, \pi}} = \chi_{\pi}^{-1}(u^{-1})$ by the characterization of reciprocity map (see Lemma \ref{reciprocitymap}). 
Consequence of above arguments, we have the following theorem.

\begin{thm} \label{Maintheorem}
We have $\mathbf{M}_{\pi, \pi'}(D_{\pi}^{\dag}) = (\RobbaB \otimes_{\mathbf{B}_{F, \pi}^{\dag}} D_{\pi}^{\dag})^{r(\pi') \otimes \chi_{\pi}^{-1}(u^{-1}) = 1}$.
\end{thm}

\section{Calculation}

If $\delta \colon F^{\times} \to F^{\times}$ is a continuous character, let $\Robba{\pi}(\delta)$ be the  $\phiqGamma{\pi}$-module over $\Robba{\pi}$ of rank 1 that has a basis $e_{\pi, \delta}$ such that $\varphi_{q} ( e_{\pi, \delta} ) = \delta ( \pi ) e_{\pi, \delta}$ and $\chi_{\pi}^{-1}(a)(e_{\pi, \delta}) = \delta(a)e_{\pi, \delta}$ with $a \in \mathcal{O}_{F}^{\times}$. 
Note that $\Robba{\pi}(\delta)$ is \'etale if and only if $v_{p}(\delta(\pi)) = 0$. From here, we assume that $\Robba{\pi}(\delta)$ is \'etale. 

To carry out calculations in this section, we often use the next lemma. 

\begin{lem} \label{solveeq}
Suppose $\varepsilon_{1} \in \mathcal{O}_{\widehat{F}_{{\rm unr}}}^{\times}$ and $\varepsilon_{2} \in \mathcal{O}_{\widehat{F}_{{\rm unr}}}$. Then there exists $x \in \mathcal{O}_{\widehat{F}_{{\rm unr}}}^{\times}$ such that $\varepsilon_{1}\sigma_{q}(x) - x = \varepsilon_{2}$. 

\end{lem}

\begin{proof}
Note that $\varepsilon_{1}\sigma_{q}(x)-x \mod \pi = \overline{\varepsilon}_{1}\overline{x}^{q} - \overline{x}$ and $\overline{\varepsilon}_{1}\overline{x}^{q} - \overline{x} = a$ with $a \in \overline{\mathbf{F}}_{q}^{\times}$ has a solution in $\overline{\mathbf{F}}_{q}^{\times}$. 
Therefore there exists $x_{0} \in \mathcal{O}_{\widehat{F}_{{\rm unr}}}^{\times}$ satisfying $\varepsilon_{1}\sigma_{q}(x_{0}) - x_{0} = \varepsilon_{2} - \pi^{n_{0}}a_{0}$ with $n_{0} > 0$ and $a_{0} \in \mathcal{O}_{\widehat{F}_{{\rm unr}}}^{\times}$. 
Inductively, let $x_{m}$ be a lift of a solution of $\overline{\varepsilon}_{1}\overline{x}^{q} - \overline{x} = \overline{a}_{m-1}$ and $n_{m}$ and $a_{m} \in \mathcal{O}_{\widehat{F}_{{\rm unr}}}$ be the elements which satisfy $\varepsilon_{1}\sigma_{q}(\pi^{n_{m-1}}x_{m-1}) - \pi^{n_{m-1}}x_{m-1} = \pi^{n_{m-1}}a_{m-1} - \pi^{n_{m}}a_{m}$. 
Because $n_{m} > n_{m-1}$, $x =\sum_{m \geq 0}\pi^{n_m}a_m$ converges and satisfies $\varepsilon_{1}\sigma_{q}(x) - x = \varepsilon_{2}$. 
\oqed\end{proof}

\begin{prop}\label{modinducedbychar}
For the \'etale $\phiqGamma{\pi}$-module $\Robba{\pi}(\delta)$, we have $\mathbf{M}_{\pi, \pi'}(\Robba{\pi}(\delta)) = \Robba{\pi'}(\delta)$. 
\end{prop}

\begin{proof}
Note that $\delta(\mathcal{O}_{F}^{\times}) \subset \mathcal{O}_{F}^{\times}$ from continuity of $\delta$. 
From Lemma \ref{solveeq}, there exists $t \in \widehat{F}_{{\rm unr}}^{\times}$ which satisfies $\sigma_{q}(t)\delta(u^{-1}) = t$, 
which means that $r(\pi') \otimes \chi_{\pi}^{-1}(u^{-1}) (t \otimes e_{\pi, \delta}) = t \otimes e_{\pi, \delta}$. 
Therefore $t \otimes e_{\pi, \delta}$ is a basis of $\mathbf{M}_{\pi, \pi'}(\Robba{\pi}(\delta))$ and it is of rank 1. 

To finish this proof, it is enough to show that the actions of $\Gamma_{\pi'}$ and $\varphi_{q}$ at $t \otimes e_{\pi, \delta}$ matches the actions at $e_{\pi', \delta}$. 
Let $a \in \mathcal{O}_{F}^{\times}$ and $\alpha$ be the lift of $\chi_{\pi'}^{-1}(a)$ in ${\rm Gal}(F_{{\rm unr}, \infty}/F)$ whose restriction to ${\rm Gal}(F_{{\rm unr}}/F)$ is trivial, which corresponds to $a^{-1}$ with the reciplosity map. 
By the property of the reciplocity map, we have $\alpha |_{\Gamma_{\pi}} = \chi_{\pi}^{-1}(a)$. 
Accordingly, the action of $\Gamma_{\pi'}$ at the basis $t \otimes e_{\pi, \delta}$ matches the action induced by $\delta$. 
In addition, we have $\varphi_{q}(t \otimes e_{\pi, \delta}) = \delta(\pi')t \otimes e_{\pi, \delta}$ by the definition of $t$ and an easy calculation. 
\oqed\end{proof}

Let $D_{\pi}^{\dag}$ be an extension of $\Robba{\pi}$ by $\Robba{\pi}(\delta)$ attached to 1-cocycle $c_{\pi} : \mathcal{O}_{F} \to \Robba {\pi}(\delta)$ in the sense of \cite[Proposition 4.1]{FX}. 
We have the exact sequence. 
\[
0 \to \Robba{\pi}(\delta) \to D_{\pi}^{\dag} \to \Robba{\pi} \to 0.
\]
Let $f_{\pi} \in D_{\pi}^{\dag}$ be a lift of $1 \in \Robba{\pi}$ which satisfies $\gamma(f_{\pi}) = f_{\pi} + c_{\pi}(\chi_{\pi}(\gamma))$ for any $\gamma \in \Gamma_{\pi}$ and $\varphi_{q}(f_{\pi}) = f_{\pi} + c_{\pi}(\pi)$. 
Then $\{ e_{\pi, \delta}, f_{\pi} \}$ is a basis of $D_{\pi}^{\dag}$. 
From Proposition \ref{modinducedbychar}, there is the element $t \otimes e_{\pi, \delta} \in \mathbf{M}_{\pi, \pi'}(D_{\pi}^{\dag})$ corresponding to $e_{\pi, \delta} \in D_{\pi}^{\dag}$. 
Accordingly, if we can find an element $f_{\pi'} \in \mathbf{M}_{\pi, \pi'}(D_{\pi}^{\dag})$ which satisfies $\{ t \otimes e_{\pi, \delta}, f_{\pi'} \}$ is linearly independent, we have $\mathbf{M}_{\pi, \pi'}(D_{\pi}^{\dag})$ is of rank 2.
In consequence the $F$-representation of $G_{F}$ attached to $D_{\pi}^{\dag}$ is not only $\pi$-overconvergent but also $\pi'$-overconvergent. 
We focus on elements of the form $1 \otimes f_{\pi} + S \otimes e_{\pi, \delta}$ with $S \in \RobbaB$. 
If $1 \otimes f_{\pi} + S \otimes e_{\pi, \delta} \in \mathbf{M}_{\pi, \pi'}(D_{\pi}^{\dag})$, then $S$ has to satisfy $r(\pi') \otimes \chi_{\pi}^{-1}(u^{-1})(S \otimes e_{\pi, \delta}) - S \otimes e_{\pi, \delta} = -c_{\pi}(u^{-1})$.

\begin{lem} \label{findseries}
There exists $S \in \RobbaB$ which satisfies $r(\pi') \otimes \chi_{\pi}^{-1}(u^{-1})(S \otimes e_{\pi, \delta}) - S \otimes e_{\pi, \delta} = -c_{\pi}(u^{-1})$. 
\end{lem}

\begin{proof}
We put $c_{\pi}(u^{-1}) = \sum_{i \in \mathbf{Z}} c_{i}u_{\pi}^{k} \otimes e_{\pi, \delta}$ with $c_{i} \in F$. 
Let $v_{{\rm min}} = \min v(c_{i})$  and $i_{0}$ be the minimal index satisfiying $v(c_{i_{0}}) = v_{{\rm min}}$. 
Note that $[u^{-1}]_{\pi}(u_{\pi})^{-1} = u_{\pi}^{-1} u ( 1 + au_{\pi} + \cdots )$ with $1+au_{\pi} + \cdots \in \mathcal{O}_{F}[\![ u_{\pi} ]\!]$. 
Applying Lemma \ref{solveeq} with $\varepsilon_{1} = u^{-i_{0}}\delta(u^{-1})$ and $\varepsilon_{2} = -\frac{c_{i_{0}}}{\pi^{v_{{\rm min}}}}$, we can find $s \in \mathcal{O}_{\widehat{F}_{{\rm unr}}}^{\times}$ satisfying $\sigma_{q}(s)\cdot u^{-i_{0}}\delta(u^{-1}) - s = -\frac{c_{i_{0}}}{\pi^{v_{{\rm min}}}}$. 
Then $r(\pi') \otimes \chi_{\pi}^{-1}(u^{-1})(\pi^{v_{{\rm min}}}su_{\pi}^{i_{0}}\otimes e_{\pi, \delta}) - \pi^{v_{{\rm min}}}su_{\pi}^{i_{0}}\otimes e_{\pi, \delta} =  ( -c_{i_{0}}u_{\pi}^{i_{0}}+u_{\pi}^{i_{0}}\pi^{v_{{\rm min}}}\alpha ) \otimes e_{\pi, \delta}$ with $\alpha \in u_{\pi}\mathcal{O}_{\widehat{F}_{{\rm unr}}}[\![ u_{\pi} ]\!]$. 
Next replacing $-c_{\pi}(u^{-1})$ by $-c_{\pi}(u^{-1}) + (c_{i_{0}}u_{\pi}^{i_{0}} - u_{\pi}^{i_{0}} \pi^{v_{{\rm min}}} \alpha) \otimes e_{\pi, \delta}$. 
Then $i_0$ or $v_{{\rm min}}$ become bigger than before. 
Therefore carrying out above calculation inductively for $i_0$ unless $v_{{\rm min}}$ does not change, $S_{v_{{\rm min}}} = \sum_{(i_{0}, s)} \pi^{v_{{\rm \min}}} s u_{\pi}^{i_{0}}$ converges in $\RobbaB$ and the minimum valuation of coefficients of $-c_{\pi}(u^{-1}) - (r(\pi') \otimes \chi_{\pi}^{-1}(u^{-1})(S_{v_{{\rm min}}} \otimes e_{\pi, \delta}) - S_{v_{{\rm min}}} \otimes e_{\pi, \delta})$ is bigger than $v_{{\rm min}}$. 

Next we want to use induction for $v_{{\rm min}}$. In this case, there remains the problem of convergency of $\sum_{v_{{\rm min}}} S_{v_{{\rm min}}}$. 
We denote $c_{\pi}(u^{-1})$ by the form $\sum_{i} \pi^i d_i (u_{\pi})$ with $d_i (u_{\pi}) \in \mathcal{O}_{F}[\![ u_{\pi} ]\!] (\frac{1}{u_{\pi}})$ whose coefficints are in $\mathcal{O}_{F}^{\times}$. 
By the definition, we have $-{\rm deg}_{u_{\pi}^{-1}} S_{v_{{\rm min}}} \geq {\rm min}_{i \leq v_{{\rm min}} } (-{\rm deg}_{u_{\pi}^{-1}} d_{i})$. 
As a result $\sum S_{v_{{\rm min}}}$ is convergent by the convergency of $\sum_{i} \pi^i d_i (u_{\pi})$. 
\oqed\end{proof}

In summary of arguments in this section, we have the next theorem.

\begin{thm}
Let $S \in \RobbaB$ be the series in Lemma \ref{findseries}. 
Then we have $\mathbf{M}_{\pi, \pi'}(D_{\pi}^{\dag}) = \Robba{\pi'}(\delta) \oplus \Robba{\pi'}(1 \otimes f_{\pi} + S \otimes e_{\pi, \delta})$. 
In particular, Galois representation attached to $D_{\pi}^{\dag}$ is also $\pi'$-overconvergent. 
\end{thm}

\end{document}